\documentclass[11pt]{amsart}

\usepackage{amsmath,amssymb}

\usepackage{amsfonts,amscd,amsxtra,calc,latexsym}

\usepackage[mathscr]{eucal}

\usepackage{ascmac,framed}

\usepackage{bm}

\usepackage{graphicx}

\usepackage[all]{xy}

\def\Spec{\mathrm{Spec}}
\def\Spf{\mathrm{Spf}}
\def\Hom{\mathrm{Hom}}
\def\Im{\mathrm{Im}}
\def\Ker{\mathrm{Ker}}
\def\Ext{\mathrm{Ext}}
\def\Coker{\mathrm{Coker}}
\def\mod{\mathrm{mod}}
\def\exp{\mathrm{exp}}

\def\ZZ{\mathbb{Z}}
\def\QQ{\mathbb{Q}}

\def\FF{\mathbb{F}}
\def\GG{\mathbb{G}}
\def\AA{\mathbb{A}}

\def\UU{\mathbb{U}}

\def\uu{\boldsymbol{u}}
\def\vv{\boldsymbol{v}}
\def\ww{\boldsymbol{w}}
\def\xx{\boldsymbol{x}}

\def\zz{\boldsymbol{z}}
\def\aa{\boldsymbol{a}}

\def\gg{\mathcal{G}}

\calclayout

\theoremstyle{plain}
\newtheorem{theorem}{Theorem}[section]
\newtheorem{remark}[theorem]{Remark}

\newtheorem{lemma}[theorem]{Lemma}
\newtheorem{corollary}[theorem]{Corollary}

\newtheorem{convention}[theorem]{Convention}

\theoremstyle{definition}

\numberwithin{equation}{section}
\numberwithin{table}{section}
\numberwithin{figure}{section}

\allowdisplaybreaks

\begin{document}

\title{On the injectivity of certain homomorphisms between extensions of $\widehat{\gg}^{(\lambda)}$ by $\widehat{\GG}_m$ over a $\ZZ_{(p)}$-algebra}
\author{Michio Amano}
\date{\today}
\address{School of Education, Meisei University, 2-1-1 Hodokubo Hino, Tokyo 191-8506, Japan}
\email{michio.amano@meisei-u.ac.jp}

\begin{abstract}
Let $\widehat{\gg}^{(\lambda)}$ be a formal group scheme which deforms $\widehat{\GG}_a$ to $\widehat{\GG}_m$. And let $\psi^{(l)}:\widehat{\gg}^{(\lambda)}\rightarrow\widehat{\gg}^{(\lambda^{p^l})}$ be the $l$-th Frobenius-type homomorphism determined by $\lambda$. We show that the homomorphism $(\psi^{(l)})^\ast:H^2_0(\widehat{\gg}^{(\lambda^{p^l})},\widehat{\GG}_m)\rightarrow H^2_0(\widehat{\gg}^{(\lambda)},\widehat{\GG}_m)$ induced by $\psi^{(l)}$ is injective over a $\ZZ_{(p)}$-algebra under a suitable restriction on $\lambda$. In this situation, the Cartier dual of $\Ker(\psi^{(l)})$, which is a finite group scheme of order $p^l$, is described over a $\ZZ/(p^n)$-algebra.
\end{abstract}

\maketitle

\section{Introduction}

Throughout this paper, we denote by $p$ a prime number. Let $A$ be a $\ZZ_{(p)}$-algebra and $\lambda$  a suitable element of $A$, where $\ZZ_{(p)}$ is a localization of rational integers $\ZZ$ at $p$. The group scheme $\gg^{(\lambda)}=\Spec\, A[X,1/(1+\lambda X)]$ which deforms the additive group scheme $\GG_a$  (in the case $\lambda=0$) to the multiplicative group scheme $\GG_m$ (in the case $\lambda\in A^\times)$ has been constructed by W.~Waterhouse and B.~Weisfeiler~\cite{WW}. Then $\gg^{(\lambda)}$ has been studied in detail independently by F.~Oort, T.~Sekiguchi and N.~Suwa~\cite{SOS} and W.~Waterhouse~\cite{W} for the purpose of unifying Artin-Schreier theory and Kummer theory. 

Let $l$ be a positive integer. For any $1\leq i\leq p^l$, there exist $k_i$ and $r_i$ uniquely such that $p^{k_i}\leq i < p^{k_i+1}$ and $i=p^{k_i}+r_i$. Note that $0\leq k_i \leq l$ for any $i=1,2,\ldots,p^l$. For each integer $0 \leq k \leq l-1$, we take $\nu_k\in A$ such that
\begin{align}
p^{l-k}\lambda^{p^k}=\nu_k\lambda^{p^l}.
\end{align}
Put $\alpha_i:=\displaystyle\binom{p^l}{i}p^{-(l-k_i)}\lambda^{r_i}\nu_{k_i}$, and set
\begin{align*}
\psi^{(l)}(X):=\displaystyle\sum_{i=1}^{p^l-1}\alpha_iX^i+X^{p^l}\in A[X].
\end{align*}
Note that, if $\lambda$ is not a zero divisor, then we can write it as
\begin{align*}
\psi^{(l)}(X)=\dfrac{1}{\lambda^{p^l}}\left\{(1+\lambda X)^{p^l}-1\right\}.
\end{align*}

\begin{remark}
{\rm
Let $\zeta_l$ be a primitive $p^l$-th root of unity. Put $\zeta_{l-k}:=\zeta_l^{p^k}$ for $1\leq k\leq l-1$. Set $A:=\ZZ_{(p)}[\zeta_l]$. For the equality
\begin{align*}
T^{p^k-1}+T^{p^k-2}+\cdots+T+1=(T-\zeta_{k})(T-\zeta_{k}^2)\cdots(T-\zeta_{k}^{p^{k}-1}),
\end{align*}
we substitute $T=1$. Then the equalities
\begin{align*}
p^{k}&=(1-\zeta_{k})^{p^{k}-1}(1+\zeta_{k})(1+\zeta_{k}+\zeta_{k}^2)\cdots(1+\zeta_{k}+\cdots+\zeta_{k}^{p^{k}-2})\\
       &=(1-\zeta_l^{p^{l-k}})^{p^{k}-1}\prod_{i=1}^{p^{k}-2}\left(1+\sum_{j=1}^{i}\zeta_{k}^j\right)\\
       &=(1-\zeta_l)^{p^{k}-1}\left(1+\sum_{i=1}^{p^{l-k}-1}\zeta_l^i\right)^{p^{k}-1}\prod_{i=1}^{p^{k}-2}\left(1+\sum_{j=1}^{i}\zeta_{k}^j\right)
\end{align*}
are obtained in $A$. Hence, when $\lambda:=1-\zeta_l\in A$, we have $p^{k}=u_{k}\lambda^{p^{k}-1}$ for $1\leq k\leq l$. Put $u_0:=1$. Therefore, for $0\leq k\leq l-1$, we can take $\nu_k:=u_l/u_k\in A$ satisfying $p^{l-k}\lambda^{p^k}=\nu_k\lambda^{p^l}$. Note that $\nu_{l-1}$ is a unit in $A$.
}
\end{remark}

For a group scheme $G$, we denote by $\widehat{G}$ the formal completion along the zero section. Since $\psi^{(l)}(X)$ satisfies $\lambda^{p^l}\psi^{(l)}(X)=(1+\lambda X)^{p^l}-1$ for any $\lambda\in A$, for an endomorphism
\begin{align*}
\varphi:\widehat{\GG}_{m,A} \rightarrow\widehat{\GG}_{m,A};\ t\mapsto \varphi(t)=t^{p^l},
\end{align*}
we obtain the commutative diagram
\begin{align*}
\begin{CD}
\widehat{\gg}^{(\lambda)}       @>{\alpha^{(\lambda)}}>>         \widehat{\GG}_{m,A} \\
                                @V{\psi^{(l)}}VV                 @VV{\varphi}V       \\
\widehat{\gg}^{(\lambda^{p^l})} @>{\alpha^{(\lambda^{p^l})}}>>   \widehat{\GG}_{m,A}, 
\end{CD}
\end{align*}
where $\alpha^{(\lambda)}(x)=1 + \lambda x$. Then
\begin{align*}
\psi^{(l)}:\widehat{\gg}^{(\lambda)}\rightarrow\widehat{\gg}^{(\lambda^{p^l})};\ x\mapsto\psi^{(l)}(x)
\end{align*}
is a well-defined surjective homomorphism. In particular, when $\lambda=0$, we can choose $\nu_k=0$ for any $k$, so
\begin{align*}
\psi^{(l)}:\widehat{\GG}_{a,A}\rightarrow\widehat{\GG}_{a,A};\ x\mapsto \psi^{(l)}(x)=x^{p^l}
\end{align*}
is nothing but the $l$-th Frobenius endomorphism over an $\FF_p$-algebra. We denote by $H^2_0(G,H)$ the Hochschild cohomology group consisting of symmetric $2$-cocycles of $G$ with coefficients in $H$ for formal group schemes $G$ and $H$ (for details, see \cite[Chap.~II.3 and Chap.~III.6]{DG}). $H^2_0(\widehat{\gg}^{(\lambda)},\widehat{\GG}_{m,A})$ has been explicitly described by T.~Sekiguchi and N.~Suwa \cite[Theorem 2.19.1.]{SS1}. More precisely, the isomorphism
\begin{align*}
\Coker[F^{(\lambda)}:W(A)\rightarrow W(A)]\xrightarrow{\sim} H^2_0(\widehat{\gg}^{(\lambda)},\widehat{\GG}_{m,A});\ \vv\mapsto F_p(\vv,\lambda;X,Y)
\end{align*}
has been constructed. Here $F_p(\vv,\lambda;X,Y)$ is a formal power series introduced in \cite{SS1}. We consider the homomorphism
\begin{align*}
(\psi^{(l)})^\ast:H^2_0(\widehat{\gg}^{(\lambda^{p^l})},\widehat{\GG}_{m,A})\rightarrow H^2_0(\widehat{\gg}^{(\lambda)},\widehat{\GG}_{m,A})
\end{align*}
induced by $\psi^{(l)}:\widehat{\gg}^{(\lambda)}\rightarrow\widehat{\gg}^{(\lambda^{p^l})}$. By the author \cite[Lemma~4.]{A1}, $(\psi^{(l)})^\ast$ has already been calculated as the pull-back by $\psi^{(l)}$. And it is given by
\begin{align*}
(\psi^{(l)})^\ast:H^2_0(\widehat{\gg}^{(\lambda^{p^l})},\widehat{\GG}_{m,A})&\rightarrow H^2_0(\widehat{\gg}^{(\lambda)},\widehat{\GG}_{m,A});\\ F_p(\vv,\lambda^{p^l};X,Y)&\mapsto F_p(\vv,\lambda^{p^l};\psi^{(l)}(X),\psi^{(l)}(Y)).
\end{align*}
Since $F^{(\lambda^{p^l})}:W_A\rightarrow W_A$ is faithfully flat (\cite[Corollary 1.8.]{SS1}), by using Lemma~4.2 in subsection 4.1, we can take $\ww\in W(B)$ such that
\begin{align*}
(\psi^{(l)})^\ast F_p(\vv,\lambda^{p^l};X,Y)=F_p(\vv,\lambda^{p^l};\psi^{(l)}(X),\psi^{(l)}(Y))=F_p(F^{(\lambda)}\circ T_{\aa}(\ww),\lambda;X,Y)
\end{align*}
for an fppf morphism $A\to B$, where $\aa\in W(A)$ such that $T_{\aa}([\lambda^{p^l}])=p^l[\lambda]$ (see subsection~2.2 for the definition of $T_{\aa}$). 

\begin{lemma}
Assume that the first component $a_0$ of $\aa$ is a unit or $0$. If
\begin{align*}
F_p(F^{(\lambda)}\circ T_{\aa}(\ww),\lambda;X,Y)=1\in H^2_0(\widehat{\gg}^{(\lambda)},\widehat{\GG}_{m,A})
\end{align*}
then there exists $E_p(\xx,\lambda^{p^l};\psi^{(l)}(T))\in A[[T]]$ such that
\begin{align*}
F_p(F^{(\lambda)}\circ T_{\aa}(\ww),\lambda;X,Y)=
\dfrac{E_p(\xx,\lambda^{p^l};\psi^{(l)}(X))\cdot E_p(\xx,\lambda^{p^l};\psi^{(l)}(Y))}{E_p(\xx,\lambda^{p^l};\psi^{(l)}(X+Y+\lambda XY))}.
\end{align*}
\end{lemma}

Our aim is to show that $(\psi^{(l)})^\ast$ is injective. This is shown by verifying that the preimage of the last equality in Lemma~1.2 is $1$ in $H^2_0(\widehat{\gg}^{(\lambda^{p^l})},\widehat{\GG}_{m,A})$ (see subsection~4.1 for details).

\begin{convention}
{\rm
Recall the elements $\nu_k$ of $A$ which we took so as satisfying (1.1). Note that they are related with the components $a_k$ of $\aa=(a_0,a_1,\dots)$.
 In particular, we have $p^l\lambda=\nu_0\lambda^{p^l}$ and $p^l\lambda=a_0\lambda^{p^l}$. So, we can and we do take so that $\nu_0=a_0$.
}
\end{convention}

The main result of this paper is:

\begin{theorem}
Let $A$ be a $\ZZ_{(p)}$-algebra. Assume that $\nu_0$ is a unit or $0$. With the above notations, the homomorphism
\begin{align*}
(\psi^{(l)})^\ast:
H^2_0(\widehat{\gg}^{(\lambda^{p^l})},\widehat{\GG}_{m,A})
\rightarrow
H^2_0(\widehat{\gg}^{(\lambda)},\widehat{\GG}_{m,A}).
\end{align*} 
is injective.
\end{theorem}

\begin{remark}
{\rm
The assumption in Theorem 1.4 that $\nu_0$ is a unit can also be seen in \cite{SOS} and \cite{W} as well as Remark~1.1. On the other hand, the assumption that $\nu_0$ is $0$ can also be seen when taking $\psi^{(l)}(X)=X^{p^l}$ over an $\FF_p$-algebra.
}
\end{remark}

The following follows from Theorem~1.4.

\begin{corollary}
{\rm
Assume that $\nu_0$ is a unit or $0$. Then we obtain the injection
\begin{align*}
(\psi^{(l)})^\ast:
H^2_0(\widehat{\GG}_{m,A},\widehat{\GG}_{m,A})
\hookrightarrow
H^2_0(\widehat{\GG}_{m,A},\widehat{\GG}_{m,A})
\end{align*} 
over a $\ZZ_{(p)}$-algebra if $\lambda\in A^\times$ and the injection
\begin{align*}
(\psi^{(l)})^\ast:
H^2_0(\widehat{\GG}_{a,A},\widehat{\GG}_{m,A})
\hookrightarrow
H^2_0(\widehat{\GG}_{a,A},\widehat{\GG}_{m,A})
\end{align*}
over an $\FF_p$-algebra if $\lambda=0$.
}
\end{corollary}

The following short exact sequence is induced by $\psi^{(l)}$:
\begin{align}
 \begin{CD}
0 @>>> N_l @>{\iota}>> \widehat{\gg}^{(\lambda)} @>{\psi^{(l)}}>> \widehat{\gg}^{(\lambda^{p^l})} @>>> 0,
   \end{CD}
\end{align}
where $N_l:=\Ker[\psi^{(l)}:\widehat{\gg}^{(\lambda)}\rightarrow\widehat{\gg}^{(\lambda^{p^l})}]$ and $\iota$ is a canonical inclusion. The exact sequence $(1.2)$ deduces the long exact sequence
\begin{align*}
 \begin{CD}
1 @>>> \Hom(\widehat{\gg}^{(\lambda^{p^l})},\widehat{\GG}_{m,A}) @>{(\psi^{(l)})^\ast}>> \Hom(\widehat{\gg}^{(\lambda)},\widehat{\GG}_{m,A})  @>{(\iota)^\ast}>>\\
 @>{(\iota)^\ast}>> \Hom(N_l,\widehat{\GG}_{m,A}) 
@>{\partial}>> \Ext^1(\widehat{\gg}^{(\lambda^{p^l})},\widehat{\GG}_{m,A}) @>{(\psi^{(l)})^\ast}>>\\
@>{(\psi^{(l)})^\ast}>> \Ext^1(\widehat{\gg}^{(\lambda)},\widehat{\GG}_{m,A}) @>>> \cdots.
   \end{CD}
\end{align*}
Since we consider only the case where the image of the boundary map $\partial$ is calculated as Hochschild extensions, we can replace $\Ext^1(\widehat{\gg}^{(\lambda)},\widehat{\GG}_{m,A})$ (resp. $\Ext^1(\widehat{\gg}^{(\lambda^{p^l})},\widehat{\GG}_{m,A})$) with $H^2_0(\widehat{\gg}^{(\lambda)},\widehat{\GG}_{m,A})$ (resp. $H^2_0(\widehat{\gg}^{(\lambda^{p^l})},\widehat{\GG}_{m,A})$) (\cite[Lemma~3.]{A1}). Therefore, the long exact sequence
\begin{align*}
\begin{CD}
1 
@>>>
\Hom(\widehat{\gg}^{(\lambda^{p^l})},\widehat{\GG}_{m,A})
@>{(\psi^{(l)})^\ast}>>
\Hom(\widehat{\gg}^{(\lambda)},\widehat{\GG}_{m,A})
@>{(\iota)^\ast}>>\\
@>{(\iota)^\ast}>>
\Hom(N_l,\widehat{\GG}_{m,A})
@>{\partial}>>
H^2_0(\widehat{\gg}^{(\lambda^{p^l})},\widehat{\GG}_{m,A})
@>{(\psi^{(l)})^\ast}>>\\
@>{(\psi^{(l)})^\ast}>>
H^2_0(\widehat{\gg}^{(\lambda)},\widehat{\GG}_{m,A})
@>>> \cdots
\end{CD}
\end{align*}
is obtained. Here we assume that $\nu_0$ is a unit or $0$. Then, by Theorem 1.4, the equalities $\Im(\partial)=\Ker (\psi^{(l)})^\ast=1$ is immediately shown. Then we have:

\begin{corollary}
Let $A$ be a $\ZZ_{(p)}$-algebra. Assume that $\nu_0$ is a unit or $0$. With the above notations, we have the short exact sequence{\rm :}
\begin{align}
\begin{CD}
1 @>>> {\rm Hom}(\widehat{\gg}^{(\lambda^{p^l})},\widehat{\GG}_{m,A}) @>{(\psi^{(l)})^\ast}>> {\rm Hom}(\widehat{\gg}^{(\lambda)},\widehat{\GG}_{m,A})@>{(\iota)^\ast}>>\\
 @>{(\iota)^\ast}>> {\rm Hom}(N_l,\widehat{\GG}_{m,A}) @>>> 1.
\end{CD}
\end{align}
\end{corollary}

\begin{remark}
{\rm
Let $A$ be a discrete valuation ring with uniformizer $\pi$ and $\mu$ an element of $(\pi)\,\backslash\,\{0\}$. The algebraic case over $A/(\mu)$ in Corollary~1.7 has already been shown by D.~Tossici \cite[Proposition 3.3.]{To}. In his proof, the nilpotency of $\pi$ in $A/(\mu)$ has been used. We prove the formal case without assuming the existence of nonzero nilpotent elements in the base ring, under the assumption that $\nu_0$ is a unit or $0$. To prove Theorem 1.4 (and Corollary~1.7), we make use of the deformations of Artin-Hasse exponential series and related formal power series introduced in \cite{SS1,SS2}.
}
\end{remark}

Using Theorem~1.4 (and Corollary~1.7), we can determine the Cartier dual of $N_l=\Ker(\psi^{(l)})$ as follows. Let $n$ be a positive integer and $A_n$ a $\ZZ/(p^n)$-algebra. Note that $\ZZ/(p^n)$-algebra $A_n$ is automatically a $\ZZ_{(p)}$-algebra, since $\ZZ/(p^n)\simeq\ZZ_{(p)}/(p^n)$. Assume that $\nu_k$'s are divisible by $p$. Then, by Lemma~4.13 in subsection 4.2, $N_l$ is a finite group scheme of order $p^l$ over $A_n$. We wish to describe the Cartier dual of $N_l$. Put $T' := F^{(\lambda)} \circ T_{\aa}$. For any $A$-algebra $B$, we set 
\begin{align*}
\left(W_A/T_{\aa}\right)(B):= \Coker [ T_{\aa} : W_A(B) \rightarrow W_A(B) ]
\end{align*}
and
\begin{align*}
\left(W_A/ T'\right)(B) := \Coker [ T' : W_A(B) \rightarrow W_A(B) ].
\end{align*}
Note that $W_A/T_{\aa}$ and $W_A/T'$ mean the cokernel in the category of presheaves on $A$-algebras. We consider the diagram
$$ \begin{CD}
 W_A(B)                @>>>          \left(W_A/T_{\aa}\right)(B)\\
 @V{F^{(\lambda)}}VV               @VV{\overline{F^{(\lambda)}}}V\\
 W_A(B)                @>>>          \left(W_A/T'\right)(B).
\end{CD} $$
Here $\overline{F^{(\lambda)}}$ is defined by $\overline{F^{(\lambda)}} (\overline{\xx}) := \overline{F^{(\lambda)}(\xx)}$. It is shown that the homomorphism $\overline{F^{(\lambda)}}$ is well-defined and that the above diagram is commutative. For any $A$-algebra $B$, we put
\begin{align*}
M_l(B):={\mathrm{Ker}} [ \overline{F^{(\lambda)}} : \left(W_A/T_{\aa}\right)(B) \rightarrow \left(W_A/ T'\right)(B) ].
\end{align*}
Note that $M_l$ means the kernel in the category of presheaves on $A$-algebras. Then we obtain:

\begin{theorem}
Let $A_n$ be a $\ZZ/(p^n)$-algebra. Assume that $\nu_k$'s are divisible by $p$ and that $\nu_0$ is $0$. Then the Cartier dual of $N_l$ is canonically isomorphic to 
\begin{align*}
\widetilde{M_l}={\mathrm{Ker}}[ \widetilde{F^{(\lambda)}} : \widetilde{W_{A_n} / T_{\aa}} \rightarrow \widetilde{W_{A_n} / T'} ],
\end{align*}
where $\widetilde{M_l}$ is the fppf-sheafification of the presheaf $M_l$ on $A_n$-algebras.
\end{theorem}

\begin{remark}
{\rm
Theorem~1.9 has been stated by the author~\cite{A2}, however, the assumption and proof were insufficient. In particular, Theorem 1.4 had been used implicitly. We correct them and give a precise proof in subsection~4.2.
}
\end{remark}

\begin{remark}
{\rm
Due to the assumption that $\nu_k$'s are divisible by $p$, we unfortunately cannot take $\nu_k$'s as stated in Remark~1.1. As another example, if $\lambda$ is a unit in $A_n$, then assuming $n\leq l$, we can take $\nu_k=p^{l-k}\lambda^{p^k}/\lambda^{p^l}$ since $\nu_0=p^l\lambda/\lambda^{p^l}=0$ and $\nu_k$'s are nilpotent elements in $A_n$.
}
\end{remark}

The contents of this paper are as follows. The next two sections are devoted to recalling the definitions and some properties of the Witt scheme and of the group scheme which deforms $\GG_a$ to $\GG_m$. In section~4 we give our proof of Theorem~1.4 and Theorem~1.9.

Throughout this paper, we use the following notations:

\begin{align*}
A\ :\ &\mbox{$\ZZ_{(p)}$-algebra},\\
A_n:\ &\mbox{$\ZZ/(p^n)$-algebra},\\
\nu_k\ :\ &\mbox{element of $A$ satisfying $p^{l-k}\lambda^{p^k}=\nu_k\lambda^{p^l}$},\\
\GG_{a,A}\ :\ & \mbox{additive group scheme over $A$},\\
\GG_{m,A}\ :\ & \mbox{multiplicative group scheme over $A$},\\
\widehat{\GG}_{a,A}\ :\ & \mbox{additive formal group scheme over $A$},\\
\widehat{\GG}_{m,A}\ :\ & \mbox{multiplicative formal group scheme over $A$},\\
W_{A}\ :\ & \mbox{group scheme of Witt vectors over $A$},\\
F\ :\ & \mbox{Frobenius endomorphism of $W_{A}$},\\
V\ :\ & \mbox{Verschiebung endomorphism of $W_{A}$}\\
[\lambda]\ :\ & \mbox{Teichm\"{u}ller lifting $(\lambda,0,0,\ldots)\in W(A)$ of $\lambda\in A$},\\
F^{(\lambda)}\ :\ & =F-[\lambda^{p-1}],\\
\aa\ :\ & \mbox{element of $W(A)$ satisfying $T_{\aa}([\lambda^{p^l}])=p^l[\lambda]$},\\
T_{\aa}\ :\ & \mbox{homomorphism decided by $\aa \in W(A)$\ (recalled in section 2)},\\
\dfrac{1}{\mu}\aa\ :\ & =\left(\dfrac{a_0}{\mu},\dfrac{a_1}{\mu},\ldots\right)\ \mbox{for}\ \mu \in A\ \mbox{and}\ \aa=(a_0,a_1,\ldots)\in W(A),\\
\Ext^1(G,H)\ :\ & \mbox{isomorphism classes of extensions of $G$ by $H$},\\
H^2_0(G,H)\ :\ & \mbox{Hochschild cohomology of $G$ with coefficients in $H$}.
\end{align*}

\section{Witt vectors}

In this short section, we recall the necessary facts on Witt vectors in this paper. For details, see \cite[Chap.~V]{DG} or \cite[Chap.~III]{HZ}.

\subsection{Definition of Witt vectors}

Let $\mathbb{X}=( X_0 , X_1 , \ldots )$ be a sequence of variables. For each $n  \geq 0 $, we denote by $\Phi_n(\mathbb{X})=\Phi_n(X_0,X_1,\ldots,X_n)$ the Witt polynomial
\begin{align*}
\Phi_n(\mathbb{X})=X_0^{p^n}+pX_1^{p^{n-1}}+\dots+p^nX_n
\end{align*}
in $\ZZ [ \mathbb{X} ] = \ZZ [ X_0 , X_1 , \ldots ]$. Let $W_{\ZZ} = \Spec\, \ZZ [\mathbb{X} ]$ be an $\infty$-dimensional affine space over $\ZZ$. The phantom map $\Phi$ is defined by
\begin{align*}
\Phi : W_{\ZZ} \rightarrow \mathbb{A}^\infty_\ZZ ; \ \xx \mapsto ( \Phi_0 ( \xx ) , \Phi_1 ( \xx ) , \ldots),
\end{align*}
where $\mathbb{A}^\infty_\ZZ$ is the usual $\infty$-dimensional affine space over $\ZZ$. The scheme $\mathbb{A}^\infty_\ZZ$ has a natural ring scheme structure. It is known that $W_{\ZZ}$ has a unique commutative ring scheme structure over $\ZZ$ such that the phantom map $\Phi$ is a homomorphism of commutative ring schemes over $\ZZ$. Then $A$-valued points $W_\ZZ(A)$ are called Witt vectors over $A$.

\subsection{Some morphisms of Witt vectors}

We define a morphism $F:W(A)\rightarrow W(A)$ by
\begin{align*}
\Phi_i(F(\xx))=\Phi_{i+1}(\xx)
\end{align*}
for $\xx\in W(A)$. If $A$ is an $\FF_p$-algebra, $F$ is nothing but the usual Frobenius endomorphism. Let $[\lambda]$ be the Teichm\"{u}ller lifting $[\lambda]=(\lambda,0,0,\ldots)\in W(A)$ for $\lambda\in A$. Then we set the endomorphism $F^{(\lambda)}:=F-[\lambda^{p-1}]$ of $W(A)$. For $\aa=(a_0,a_1,\ldots) \in W(A)$, we also define a morphism $T_{\aa}:W(A) \rightarrow W(A)$ by
\begin{align*}
\Phi_n ( T_{\aa} (\xx) ) = {a_0}^{p^n} \Phi_n (\xx) + p {a_1}^{p^{n-1}} \Phi_{n-1} ( \xx ) + \cdots + p^n a_n \Phi_0 ( \xx )
\end{align*}
for $\xx \in W(A)$ (\cite[Chap.4, p.20]{SS2}). Then the morphism $T_{\aa}$ is called $T$-map.

\section{The group scheme $\gg^{(\lambda)}$ which deforms $\GG_a$ to $\GG_m$}

In this short section, we recall necessary facts on the group scheme $\gg^{(\lambda)}$ which deforms $\GG_a$ to $\GG_m$ for this paper.

\subsection{Definition of the group scheme $\gg^{(\lambda)}$}

Let $A$ be a ring and $\lambda$ an element of $A$. Put $\gg^{(\lambda)} := \Spec\, A [ X , 1 / (1 + \lambda X ) ]$. We define a morphism $\alpha^{(\lambda)}$ by
\begin{align*}
\alpha^{(\lambda)} : \gg^{(\lambda)} \rightarrow \GG_{m,A};\ 
x \mapsto 1 + \lambda x.
\end{align*}
It is known that $\gg^{(\lambda)}$ has a unique commutative group scheme structure such that $\alpha^{(\lambda)}$ is a group scheme homomorphism over $A$. Then the group scheme structure of $\gg^{(\lambda)}$ is given by $x \cdot y = x + y + \lambda xy$.
If $\lambda$ is invertible in $A$, $\alpha^{(\lambda)}$ is an $A$-isomorphism. On the other hand, if $\lambda=0$, $\gg^{(\lambda)}$ is nothing but the additive group scheme $\GG_{a,A}$.

\subsection{Deformed Artin-Hasse exponential series}

In this subsection, we recall several formal series and relations between them, which are introduced and proved in \cite{SS1} and \cite{SS2}.

The Artin-Hasse exponential series $E_p(X)$ is given by
\begin{align*}
E_p (X) = \exp \left( \sum_{r \geq 0} \frac{X^{p^r}}{p^r} \right)
\in \ZZ_{(p)} [[X]].
\end{align*}
We define a formal power series $E_p( U, \Lambda ; X )$ in $\QQ [ U, \Lambda ] [[X]]$ by
\begin{align*}
E_p ( U, \Lambda ; X ) = ( 1 + \Lambda X)^{\frac{U}{\Lambda}} \prod_{k=1}^{\infty} ( 1 + \Lambda^{p^k} X^{p^k} )^{ \frac{1}{p^k} \left\{\left( \frac{U}{\Lambda} \right)^{p^k}-\left( \frac{U}{\Lambda} \right)^{ p^{k-1} } \right\} }.
\end{align*}
As in \cite[Corollary~2.5.]{SS1} or \cite[Lemma~4.8.]{SS2}, we see that the formal power series $E_p(U,\Lambda;X)$ is integral over $\ZZ_{(p)}$. Note that $E_p(1,0;X)=E_p(X)$. Let $A$ be a $\ZZ_{(p)}$-algebra. For $\lambda\in A$ and $\vv=(v_0,v_1,\ldots)\in W(A)$, we define a formal power series $E_p(\vv,\lambda;X)$ in $A[[X]]$ by
\begin{align*}
E_p(\vv,\lambda;X)=\prod_{k=0}^{\infty}E_p(v_k,\lambda^{{p^k}};X^{p^k})
                  =(1+\lambda X)^{\frac{v_0}{\lambda}}\prod_{k=1}^{\infty}(1+\lambda^{p^k}X^{p^k})^{\frac{1}{p^k\lambda^{p^k}}\Phi_{k-1}(F^{(\lambda)}(\vv))}.
\end{align*}
And we define a formal power series $F_p(\vv,\lambda;X,Y)$ as follows:
\begin{align*}
F_p(\vv,\lambda;X,Y)=\prod_{k=1}^{\infty}\left(\frac{(1+\lambda^{p^k}X^{p^k})(1+\lambda^{p^k}Y^{p^k})}{1+\lambda^{p^k}(X+Y+\lambda XY)^{p^k}}\right)^{\frac{1}{p^k\lambda^{p^k}}\Phi_{k-1}(\vv)}.
\end{align*}
As in \cite[Lemma~2.16.]{SS1} or \cite[Lemma~4.9.]{SS2}, we see that the formal power series $F_p(\vv,\lambda;X,Y)$ is integral over $\ZZ_{(p)}$. For the formal power series $F_p(F^{(\lambda)}(\vv),\lambda;X,Y)$, the equalities
\begin{align*}
F_p(F^{(\lambda)}(\vv),\lambda;X,Y)
&=\prod_{k=1}^{\infty}\left(\dfrac{(1+\lambda^{p^k}X^{p^k})(1+\lambda^{p^k}Y^{p^k})}{1+\lambda^{p^k}(X+Y+\lambda XY)^{p^k}}\right)^{\frac{1}{p^k\lambda^{p^k}}\Phi_{k-1}(F^{(\lambda)}(\vv))}\\
&=\dfrac{E_p(\vv,\lambda;X)E_p(\vv,\lambda;Y)}{E_p(\vv,\lambda;X+Y+\lambda XY)}
\end{align*}
hold (for details, see \cite[proof of Lemma~3.1]{SS1}). For $\uu\in W(A)$, we define
\begin{align*}
\tilde{p}^kE_p(\uu,\lambda;X):=E_p(V^k(\uu^{(p^k)}),\lambda;X).
\end{align*}
Then the formal power series $G_p(\vv,\mu;E)$ is introduced as follows:
\begin{align*}
G_p(\vv,\mu;E)&=\prod_{k \geq 1}\left(\frac{1+(E-1)^{p^k}}{\tilde{p}^kE} \right)^{\frac{1}{p^k\mu^{p^k}}\Phi_{k-1}(\vv)},
\end{align*}
where $\mu\in A$ and $E:=E_p(\uu,\lambda;X)$ for a vector $\uu\in W(A)$.

\section{Proof of Theorem~1.4 and Theorem~1.9}

In this section, we prove Theorem 1.4 and Theorem~1.9.

\subsection{Proof of Theorem~1.4}

The notation described in the introduction will be taken over. We denote by $Z^2(\widehat{\gg}^{(\lambda)},\widehat{\GG}_{m,A})$ the subgroup of $A[[X,Y]]^\times$ formed by the symmetric $2$-cocycles of $\widehat{\gg}^{(\lambda)}$ with coefficients in $\widehat{\GG}_{m,A}$. Take $F(T)\in A[[T]]^\times$. Then $F(X)F(Y)F(X+Y+\lambda XY)^{-1}\in Z^2(\widehat{\gg}^{(\lambda)},\widehat{\GG}_{m,A})$. Moreover we set
\begin{align}
B^2(\widehat{\gg}^{(\lambda)},\widehat{\GG}_{m,A}):=\left\{F(X)F(Y)F(X+Y+\lambda XY)^{-1}\mid F(T)\in A[[T]]^\times \right\}.
\end{align}
Put
\begin{align*}
H^2_0(\widehat{\gg}^{(\lambda)},\widehat{\GG}_{m,A})
:=Z^2(\widehat{\gg}^{(\lambda)},\widehat{\GG}_{m,A})/B^2(\widehat{\gg}^{(\lambda)},\widehat{\GG}_{m,A})
\end{align*}
(for details, see \cite[subsection 2.2]{SS1}). In \cite[Theorem 2.19.1.]{SS1}, the isomorphism
\begin{align}
\Coker[F^{(\lambda)}:W(A)\rightarrow W(A)]\xrightarrow{\sim} H^2_0(\widehat{\gg}^{(\lambda)},\widehat{\GG}_{m,A});\ \vv\mapsto F_p(\vv,\lambda;X,Y)
\end{align}
has been given. We consider the homomorphism
\begin{align*}
(\psi^{(l)})^\ast:H^2_0(\widehat{\gg}^{(\lambda^{p^l})},\widehat{\GG}_{m,A})&\rightarrow H^2_0(\widehat{\gg}^{(\lambda)},\widehat{\GG}_{m,A});\\ F_p(\vv,\lambda^{p^l};X,Y)&\mapsto F_p(\vv,\lambda^{p^l};\psi^{(l)}(X),\psi^{(l)}(Y)).
\end{align*}
induced by $\psi^{(l)}:\widehat{\gg}^{(\lambda)}\rightarrow\widehat{\gg}^{(\lambda^{p^l})}$ (\cite[Lemma~4.]{A1}). We wish to prove that $(\psi^{(l)})^\ast$ is injective. We choose $\aa\in W(A)$ such that $T_{\aa}([\lambda^{p^l}])=p^l[\lambda]$, where $T_{\aa}([\lambda^{p^l}])=(\lambda^{p^l}a_0,\lambda^{p^l}a_1,\ldots)$ for $\aa=(a_0,a_1,\ldots)$. Note that, if $\lambda$ is not a zero divisor, then we can write as $\aa=p^l[\lambda]/\lambda^{p^l}$, where $p^l[\lambda]/\lambda^{p^l}=(b_0/\lambda^{p^l},b_1/\lambda^{p^l},\ldots)$ for $p^l[\lambda]=(b_0,b_1,\ldots)$.

\begin{lemma}
Put $E:=E_p([\lambda^{p^l}],\lambda^{p^l};\psi^{(l)}(X))$. For any $\xx\in W(A)$, the equality
\begin{align*}
G_p(\xx,\lambda^{p^l};E)=1
\end{align*}
holds over $A$.
\end{lemma}

\begin{proof}
Since $E=E_p([\lambda^{p^l}],\lambda^{p^l};\psi^{(l)}(X))=1+\lambda^{p^l}\psi^{(l)}(X)$, the equality
\begin{align*}
1+(E-1)^{p^k}=1+\lambda^{p^{l+k}} \psi^{(l)}(X)^{p^k}
\end{align*}
holds. On the other hand, by \cite[Lemma~4.10.]{SS2}, we have
\begin{align*}
\widetilde{p}^{k}E=E_p([\lambda^{p^{l+k}}],\lambda^{p^{l+k}};\psi^{(l)}(X)^{p^k})=1+\lambda^{p^{l+k}} \psi^{(l)}(X)^{p^k}.
\end{align*}
Hence the equalities
\begin{align*}
G_p(\xx,\lambda^{p^l};E)
=\prod_{k \geq 1}\left(\frac{1+(E-1)^{p^k}}{\widetilde{p}^{k}E} \right)^{\frac{1}{p^k\lambda^{p^k}}\Phi_{k-1}(\xx)}=1
\end{align*}
are obtained.
\end{proof}

The following lemma plays a key role in our proof:

\begin{lemma}
For any $\vv\in W(A)$, the equality
\begin{align*}
E_p(\vv,\lambda^{p^l};\psi^{(l)}(T))=E_p(T_{\aa}(\vv),\lambda;T)
\end{align*}
holds over $A$.
\end{lemma}

\begin{proof}
By the definition and the results stated in \cite[Chap.4,\ p.29 and Proposition~4.11.]{SS2}, we have the equalities
\begin{align*}
E_p(\UU,\Lambda^{p^l};\psi^{(l)}(X))&=E_p(\UU,\Lambda^{p^l};\{(1+\Lambda X)^{p^l}-1\}/\Lambda^{p^l})\\
                                    &=E_p(\UU,\Lambda^{p^l};(E_p([\Lambda],\Lambda;X)^{p^l}-1)/\Lambda^{p^l})\\
                                    &=\widetilde{E}_p(\UU,\Lambda^{p^l};E_p(p^l[\Lambda],\Lambda;X))\cdot G_p(F^{(\Lambda^{p^l})}\UU,\Lambda^{p^l};E_p(p^l[\Lambda],\Lambda;X))\\
                                    &=E_p(T_{p^l[\Lambda]/\Lambda^{p^l}}\UU,\Lambda;X)\cdot G_p(F^{(\Lambda^{p^l})}\UU,\Lambda^{p^l};E_p([\Lambda^{p^l}],\Lambda^{p^l};\psi^{(l)}(X)))
\end{align*}
for a variable $\Lambda$ and a vector $\UU:=(U_0,U_1,\ldots)$. Here, using Lemma~4.1, we obtain
\begin{align*}
E_p(\UU,\Lambda^{p^l};\psi^{(l)}(X))&=E_p(T_{p^l[\Lambda]/\Lambda^{p^l}}\UU,\Lambda;X)\in\ZZ_{(p)}[\UU,\Lambda,\dfrac{p^l[\Lambda]}{\Lambda^{p^l}}][[X]]
\end{align*}
By considering $p^l[\Lambda]/\Lambda^{p^l}$ as a vector $\AA$ such that $T_\AA([\Lambda^{p^l}])=p^l[\Lambda]$, we can see that the above last equality is also true for $\vv$, $\lambda$ and $\aa$.
\end{proof}

Note that, for the case $\vv\in\Ker(F^{(\lambda^{p^l})})(A)$, Lemma~4.2 has already been stated in the proof in \cite[Lemma~1.]{A1}. By Lemma~4.2, we can rewrite the cocycle $F_p(\vv,\lambda^{p^l};\psi^{(l)}(X),\psi^{(l)}(Y))$ as follows.

\begin{lemma}
Assume that for $\vv\in W(A)$, there exists $\ww\in W(B)$ satisfying $\vv=F^{(\lambda^{p^l})}(\ww)$, where $B$ is an $A$-algebra. Then the equality
\begin{align*}
F_p(\vv,\lambda^{p^l};\psi^{(l)}(X),\psi^{(l)}(Y))=F_p(F^{(\lambda)}\circ T_{\aa}(\ww),\lambda;X,Y)
\end{align*}
holds.
\end{lemma}

\begin{proof}
By Lemma 4.2, the equalities
\begin{align*}
&F_p(\vv,\lambda^{p^l};\psi^{(l)}(X),\psi^{(l)}(Y))\\
&=F_p(F^{(\lambda^{p^l})}(\ww),\lambda^{p^l};\psi^{(l)}(X),\psi^{(l)}(Y))\\
&=\dfrac{E_p(\ww,\lambda^{p^l};\psi^{(l)}(X))\cdot E_p(\ww,\lambda^{p^l};\psi^{(l)}(Y))}{E_p(\ww,\lambda^{p^l};\psi^{(l)}(X)+\psi^{(l)}(Y)+\lambda^{p^l}\psi^{(l)}(X)\psi^{(l)}(Y))}\\
&=\dfrac{E_p(\ww,\lambda^{p^l};\psi^{(l)}(X))\cdot E_p(\ww,\lambda^{p^l};\psi^{(l)}(Y))}{E_p(\ww,\lambda^{p^l};\psi^{(l)}(X+Y+\lambda XY))}\\
&=\dfrac{E_p(T_{\aa}(\ww),\lambda;X)\cdot E_p(T_{\aa}(\ww),\lambda;Y)}{E_p(T_{\aa}(\ww),\lambda;X+Y+\lambda XY)}\\
&=F_p(F^{(\lambda)}\circ T_{\aa}(\ww),\lambda;X,Y)
\end{align*}
hold.
\end{proof}

The following Lemma is used in our proof of the case where $\nu_0$ is a unit.

\begin{lemma}
The group homomorphism
\begin{align*}
(\psi^{(l)})^\ast:Z^2(\widehat{\gg}^{(\lambda^{p^l})},\widehat{\GG}_{m,A})\to Z^2(\widehat{\gg}^{(\lambda)},\widehat{\GG}_{m,A});\ 
F(X,Y)\mapsto F(\psi^{(l)}(X),\psi^{(l)}(Y))
\end{align*}
induced by $\psi^{(l)}:\widehat{\gg}^{(\lambda)}\rightarrow\widehat{\gg}^{(\lambda^{p^l})}$ is injective.
\end{lemma}

\begin{proof}
Note that the coordinate rings of $\widehat{\gg}^{(\lambda)}$ and $\widehat{\gg}^{(\lambda^{p^l})}$ are given by
\begin{align*}
\mathcal{O}(\widehat{\gg}^{(\lambda)})=\mathcal{O}(\widehat{\gg}^{(\lambda^{p^l})})=A[[T]].
\end{align*}
Since $\psi^{(l)}$ is surjective, we have the injective homomorphism
\begin{align*}
A[[T]]\hookrightarrow A[[T]];\ T\mapsto \psi^{(l)}(T).
\end{align*}
Hence, by the correspondence
\begin{align*}
F(X,Y)\mapsto F(\psi^{(l)}(X),\psi^{(l)}(Y)),
\end{align*}
we obtain the injection
\begin{align*}
\mathcal{O}(\widehat{\gg}^{(\lambda^{p^l})}\times\widehat{\gg}^{(\lambda^{p^l})})=A[[X,Y]]
\hookrightarrow
\mathcal{O}(\widehat{\gg}^{(\lambda)}\times\widehat{\gg}^{(\lambda)})=A[[X,Y]].
\end{align*}
This leads to the injection
\begin{align*}
(\psi^{(l)})^\ast:
&\Hom_{A\mathchar`-\mathrm{alg}}(\widehat{\gg}^{(\lambda^{p^l})}\times\widehat{\gg}^{(\lambda^{p^l})},\widehat{\GG}_{m,A})=A[[X,Y]]^\times\\
&\hookrightarrow
\Hom_{A\mathchar`-\mathrm{alg}}(\widehat{\gg}^{(\lambda)}\times\widehat{\gg}^{(\lambda)},\widehat{\GG}_{m,A})=A[[X,Y]]^\times.
\end{align*}
Therefore, the injection
\begin{align*}
(\psi^{(l)})^\ast:Z^2(\widehat{\gg}^{(\lambda^{p^l})},\widehat{\GG}_{m,A})
\hookrightarrow
Z^2(\widehat{\gg}^{(\lambda)},\widehat{\GG}_{m,A}).
\end{align*}
is obtained.
\end{proof}

The following Lemma is used in our proof of the case $\nu_0=0$.

\begin{lemma}
Let $A\to B$ be an fppf morphism. Assume that
\begin{align*}
E_p(\vv,\lambda^{p^l};\psi^{(l)}(T))\in A[[T]]\quad \mbox{and}\quad
E_p(\vv,\lambda^{p^l};\psi^{(l)}(T))\in B[[\psi^{(l)}(T)]].
\end{align*}
Then $E_p(\vv,\lambda^{p^l};\psi^{(l)}(T))$ is contained in $A[[\psi^{(l)}(T)]]$.
\end{lemma}

\begin{proof} 
Put
\begin{align*}
E_p(\vv,\lambda^{p^l};\psi^{(l)}(T)):=d_0+d_1\psi^{(l)}(T)+d_2\psi^{(l)}(T)^2+\cdots+d_k\psi^{(l)}(T)^k+\cdots.
\end{align*}
We show $d_k\in A$ by induction on $k$. If $k=0$, $d_0=1\in A$ holds. Assume $d_i\in A$ for $1\leq i\leq k-1$. Then 
 we have
\begin{align*}
d_0+d_1\psi^{(l)}(T)+d_2\psi^{(l)}(T)^2+\cdots+d_{k-1}\psi^{(l)}(T)^{k-1}\in A[T].
\end{align*}
Hence the formal power series
\begin{align*}
f_k(\psi^{(l)}(T))
&:=E_p(\vv,\lambda^{p^l};\psi^{(l)}(T))\\
&\quad -\left\{d_0+d_1\psi^{(l)}(T)+d_2\psi^{(l)}(T)^2+\cdots+d_{k-1}\psi^{(l)}(T)^{k-1}\right\}
\end{align*}
 is contained in $A[[T]]$ and $(\psi^{(l)}(T)^k)$, where $(\psi^{(l)}(T)^k)$ is the ideal of $B[[T]]$ generated by $\psi^{(l)}(T)^k$. Here, since $A\to B$ is faithfully flat, we obtain the injection
\begin{align*}
A[[T]]/(\psi^{(l)}(T)^k)\hookrightarrow B[[T]]/(\psi^{(l)}(T)^k).
\end{align*}
This leads to 
\begin{align*}
f_k(\psi^{(l)}(T))=(d_k+d_{k+1}\psi^{(l)}(T)+\cdots)\psi^{(l)}(T)^k\in(\psi^{(l)}(T)^k),
\end{align*}
where $(\psi^{(l)}(T)^k)$ is the ideal of $A[[T]]$ generated by $\psi^{(l)}(T)^k$. Hence, since $\psi^{(l)}(T)$ has no constant term, we obtain $d_k\in A$. Therefore, since $d_k\in A$ for any $k$,
\begin{align*}
E_p(\vv,\lambda^{p^l};\psi^{(l)}(T))\in A[[\psi^{(l)}(T)]]
\end{align*}
is obtained. 
\end{proof}

With the above preparations, we complete our proof of Theorem 1.4. Take
\begin{align*}
F_p(\vv,\lambda^{p^l};X,Y)\in Z^2(\widehat{\gg}^{(\lambda^{p^l})},\widehat{\GG}_{m,A})
\end{align*}
as a representative of an element of $H^2_0(\widehat{\gg}^{(\lambda^{p^l})},\widehat{\GG}_{m,A})$, where $\vv\in W(A)$. Then we have the cocycle
\begin{align*}
(\psi^{(l)})^\ast F_p(\vv,\lambda^{p^l};X,Y)=F_p(\vv,\lambda^{p^l};\psi^{(l)}(X),\psi^{(l)}(Y))\in Z^2(\widehat{\gg}^{(\lambda)},\widehat{\GG}_{m,A}).
\end{align*}
Note that, since $F^{(\lambda^{p^l})}:W_A\rightarrow W_A$ is faithfully flat (\cite[Corollary 1.8.]{SS1}), for any $\vv\in W(A)$, there exists $\ww\in W(B)$ such that $\vv=F^{(\lambda^{p^l})}(\ww)$ for an fppf morphism $A\to B$. Then, by Lemma~4.3, the equality
\begin{align*}
F_p(\vv,\lambda^{p^l};\psi^{(l)}(X),\psi^{(l)}(Y))
=F_p(F^{(\lambda)}\circ T_{\aa}(\ww),\lambda;X,Y)
\end{align*}
holds in $Z^2(\widehat{\gg}^{(\lambda)},\widehat{\GG}_{m,A})$ over $A$. Assume
\begin{align*}
F_p(F^{(\lambda)}\circ T_{\aa}(\ww),\lambda;X,Y)\in B^2(\widehat{\gg}^{(\lambda)},\widehat{\GG}_{m,A}).
\end{align*}
By the isomorphism (4.2), this assumption is equivalent to the existence of $\uu\in W(A)$ such that $F^{(\lambda)}\circ T_{\aa}(\ww)=F^{(\lambda)}(\uu)$. We now prove Theorem~1.4 for the cases where $\nu_0$ is a unit and $\nu_0=0$, respectively. First, we show the case where $\nu_0$ is a unit. We can take the first component $a_0$ of $\aa$ such that $a_0=\nu_0$. In this case, since $T_{\aa}:W(A)\rightarrow W(A)$ is an isomorphism by \cite[Proposition~4.4.]{SS2} or \cite[Lemma~6.1.4.]{MRT1}, there exists $\uu'\in W(A)$ such that $\uu=T_{\aa}(\uu')$. Then we have the equalities
\begin{align*}
F_p(\vv,\lambda^{p^l};\psi^{(l)}(X),\psi^{(l)}(Y))
&=F_p(F^{(\lambda)}\circ T_{\aa}(\ww),\lambda;X,Y)\\
&=F_p(F^{(\lambda)}(\uu),\lambda;X,Y)\\
&=F_p(F^{(\lambda)}\circ T_{\aa}(\uu'),\lambda;X,Y)\\
&=F_p(F^{(\lambda^{p^l})}(\uu'),\lambda^{p^l};\psi^{(l)}(X),\psi^{(l)}(Y)).
\end{align*}
Here, by Lemma~4.4, we obtain the equalities
\begin{align*}
F_p(\vv,\lambda^{p^l};X,Y)&=F_p(F^{(\lambda^{p^l})}(\uu'),\lambda^{p^l};X,Y)\\
                          &=\dfrac{E_p(\uu',\lambda^{p^l};X)\cdot E_p(\uu',\lambda^{p^l};Y)}{E_p(\uu',\lambda^{p^l};X+Y+\lambda^{p^l} XY)}\in B^2(\widehat{\gg}^{(\lambda^{p^l})},\widehat{\GG}_{m,A})
\end{align*}
over $A$. This means that $(\psi^{(l)})^\ast$ is injective. Next, we show the case $\nu_0=0$. Then we can take as $a_0=0$. In this case, since the first component of $T_{\aa}(\ww)$ is given by $a_0w_0=0$ for $\ww=(w_0,w_1,\ldots)$, we have
\begin{align}
\begin{split}
E_p(T_{\aa}(\ww),\lambda;T)
&=(1+\lambda T)^{\frac{a_0w_0}{\lambda}}\prod_{k=1}^{\infty}(1+\lambda^{p^k}T^{p^k})^{\frac{1}{p^k\lambda^{p^k}}\Phi_{k-1}(F^{(\lambda)}\circ T_{\aa}(\ww))}\\
&=\prod_{k=1}^{\infty}(1+\lambda^{p^k}T^{p^k})^{\frac{1}{p^k\lambda^{p^k}}\Phi_{k-1}(F^{(\lambda)}(\uu))}\\
&=E_p(\uu,\lambda;T)\cdot(1+\lambda T)^{-\frac{u_0}{\lambda}}\in A[[T]]^\times,
\end{split}
\end{align}
where $\uu=(u_0,u_1,\ldots)\in W(A)$. We aim at proving
\begin{align*}
F_p(\vv,\lambda^{p^l};X,Y)
=\dfrac{E_p(\ww,\lambda^{p^l};X)\cdot E_p(\ww,\lambda^{p^l};Y)}{E_p(\ww,\lambda^{p^l};X+Y+\lambda XY)}\in B^2(\widehat{\gg}^{(\lambda^{p^l})},\widehat{\GG}_{m,A})
\end{align*}
over $A$. By Lemma 4.2 and (4.3), we obtain
\begin{align}
E_p(\ww,\lambda^{p^l};\psi^{(l)}(T))=E_p(T_{\aa}(\ww),\lambda;T)\in A[[T]]^\times.
\end{align}
On the other hand, since $\ww\in W(B)$, we have $E_p(\ww,\lambda^{p^l};T)\in B[[T]]^\times$. This implies
\begin{align}
E_p(\ww,\lambda^{p^l};\psi^{(l)}(T))\in B[[\psi^{(l)}(T)]]^\times.
\end{align}
Then, applying Lemma 4.5 to (4.4) and (4.5), we get
\begin{align*}
E_p(\ww,\lambda^{p^l};T)\in A[[T]]^\times.
\end{align*}
Therefore, by (4.1), the cocycle
\begin{align*}
F_p(\vv,\lambda^{p^l};X,Y)
=F_p(F^{(\lambda^{p^l})}(\ww),\lambda^{p^l};X,Y)
=\dfrac{E_p(\ww,\lambda^{p^l};X)\cdot E_p(\ww,\lambda^{p^l};Y)}{E_p(\ww,\lambda^{p^l};X+Y+\lambda^{p^l} XY)}
\end{align*}
is contained in $B^2(\widehat{\gg}^{(\lambda^{p^l})},\widehat{\GG}_{m,A})$ over $A$. This means that $(\psi^{(l)})^\ast$ is injective. Our proof of Theorem~1.4 is now complete. 

\begin{remark}
{\rm
By A.~M\'{e}zard, M.~Romagny and D.~Tossici \cite[Lemma~6.3.]{MRT2}, the existence of $T'_{\aa}$ which is $F^{(\lambda)}\circ T_{\aa}=T'_{\aa}\circ F^{(\lambda^{p^l})}$ has been shown. Here, under the assumption that $\nu_0=a_0$ is a unit or $\nu_0=a_0=0$, we prove that
\begin{align}
T'_{\aa}:\Coker(F^{(\lambda^{p^l})})(A)\to \Coker(F^{(\lambda)})(A)
\end{align}
is injective. For any $F^{(\lambda^{p^l})}(\vv)\in \Im(F^{(\lambda^{p^l})})(A)$, since
\begin{align*}
T'_{\aa}\circ F^{(\lambda^{p^l})}(\vv)=F^{(\lambda)}\circ T_{\aa}(\vv)\in\Im(F^{(\lambda)})(A),
\end{align*}
(4.6) is well-defined. We consider the diagram
\begin{align}
\begin{CD}
H^2_0(\widehat{\gg}^{(\lambda^{p^l})},\widehat{\GG}_{m,A}) @>{(\psi^{(l)})^\ast}>> H^2_0(\widehat{\gg}^{(\lambda)},\widehat{\GG}_{m,A})\\
@A{\phi_1}AA @A{\phi_2}AA\\
\Coker(F^{(\lambda^{p^l})})(A) @>{T'_{\aa}}>> \Coker(F^{(\lambda)})(A),
\end{CD}
\end{align}
where $\phi_1$ and $\phi_2$ are isomorphisms mentioned in (4.2). Let $\vv\in W(A)$ be a representative of $\overline{\vv}\in\Coker(F^{(\lambda^{p^l})})(A)$. Then, by Lemma~4.3, we have the equalities
\begin{align*}
(\psi^{(l)})^\ast\circ \phi_1(\vv)&=F_p(F^{(\lambda)}\circ T_{\aa}(\ww),\lambda;X,Y)\\
                                &=F_p(T'_{\aa}\circ F^{(\lambda^{p^l})}(\ww),\lambda;X,Y)\\
                                &=F_p(T'_{\aa}(\vv),\lambda;X,Y)=\phi_2\circ T'_{\aa}(\vv).
\end{align*}
Hence the diagram (4.7) is commutative. Therefore, since $(\psi^{(l)})^\ast$ is injective under our assumptions, (4.6) is also injective.
}
\end{remark}

\subsection{Proof of Theorem~1.9}

The notation described in the introduction will be taken over. Let $B$ be an $A$-algebra.

\begin{lemma}
The equality
\begin{align*}
\Ker(F^{(\lambda^{p^l})})(B)=\Ker(F^{(\lambda)}\circ T_{\aa})(B)
\end{align*}
holds.
\end{lemma}

\begin{proof}
Note that the morphism
\begin{align}
\Ker[F^{(\lambda)}:W(B)\to W(B)]\to \Hom(\widehat{\gg}^{(\lambda)},\widehat{\GG}_{m,B});\ \vv\mapsto E_p(\vv,\lambda;X)
\end{align}
is an isomorphism (\cite[Theorem~2.19.1.]{SS1}). We consider the diagram
\begin{align}
\begin{CD}
\Hom(\widehat{\gg}^{(\lambda^{p^l})},\widehat{\GG}_{m,B}) @>{(\psi^{(l)})^\ast}>> \Hom(\widehat{\gg}^{(\lambda)},\widehat{\GG}_{m,B})\\
@A{\phi_3}AA @A{\phi_4}AA\\
\Ker(F^{(\lambda^{p^l})})(B) @>{T_{\aa}}>> \Ker(F^{(\lambda)})(B),
\end{CD}
\end{align}
where $\phi_3$ and $\phi_4$ are isomorphisms mentioned in (4.8). Take $\vv\in\Ker(F^{(\lambda^{p^l})})(B)$. By Lemma~4.2, we have
\begin{align*}
(\psi^{(l)})^\ast\circ \phi_3(\vv)&=E_p(\vv,\lambda^{p^l};\psi^{(l)}(X))\\
                                  &=E_p(T_{\aa}(\vv),\lambda;X)\in\Hom(\widehat{\gg}^{(\lambda)},\widehat{\GG}_{m,B}).
\end{align*}
Here, since $\phi_4:\Ker(F^{(\lambda)})(B)\xrightarrow{\sim}\Hom(\widehat{\gg}^{(\lambda)},\widehat{\GG}_{m,B})$, we obtain $T_{\aa}(\vv)\in\Ker(F^{(\lambda)})(B)$. Hence
\begin{align*}
T_{\aa}:\Ker(F^{(\lambda^{p^l})})(B) \to \Ker(F^{(\lambda)})(B)
\end{align*}
is well-defined and the diagram (4.9) is commutative. Then, for any $\vv\in\Ker(F^{(\lambda^{p^l})})(B)$, we have $T_{\aa}(\vv)\in\Ker(F^{(\lambda)})(B)$. This means $F^{(\lambda)}\circ T_{\aa}(\vv)=0$ for any $\vv\in\Ker(F^{(\lambda^{p^l})})(B)$. Therefore, we obtain the inclusion $\Ker(F^{(\lambda^{p^l})})(B)\subset\Ker(F^{(\lambda)}\circ T_{\aa})(B)$. To show the reverse inclusion, take $\vv\in\Ker(F^{(\lambda)}\circ T_{\aa})(B)$. Then, since $F^{(\lambda)}\circ T_{\aa}(\vv)=0$, we have the equalities
\begin{align*}
F_p(F^{(\lambda)}\circ T_{\aa}(\vv),\lambda;X,Y)=F_p(F^{(\lambda^{p^l})}(\vv),\lambda^{p^l};\psi^{(l)}(X),\psi^{(l)}(Y))=1.
\end{align*} 
By Lemma~4.4, $F_p(F^{(\lambda^{p^l})}(\vv),\lambda^{p^l};X,Y)=1$ is obtained. This means
\begin{align*}
E_p(\vv,\lambda^{p^l};X)E_p(\vv,\lambda^{p^l};Y)=E_p(\vv,\lambda^{p^l};X+Y+\lambda^{p^l} XY).
\end{align*}
Therefore, since $E_p(\vv,\lambda^{p^l};T)\in\Hom(\widehat{\gg}^{(\lambda^{p^l})},\widehat{\GG}_{m,B})$, we have $\vv\in\Ker(F^{(\lambda^{p^l})})(B)$ by $\phi_3$.
\end{proof}

\begin{remark}
{\rm
Lemma 4.7 has been shown by the author~\cite[Lemma~1]{A2} (and \cite{A3}), by direct calculations. But there, it was used to imply that $a_0$ is not a zero divisor.
}
\end{remark}

\begin{lemma}
We have the short exact sequence:
\begin{align}
\begin{CD}
\Ker(F^{(\lambda^{p^l})})(B) @>{T_{\aa}}>> \Ker(F^{(\lambda)})(B) @>{\pi}>> M_l(B) @>>> 0,
\end{CD}
\end{align}
where $\pi$ is a homomorphism induced by the natural projection $W_A(B) \twoheadrightarrow (W_A/T_{\aa})(B)$.
\end{lemma}

\begin{proof}
We show that $\Im(T_{\aa})=\Ker(\pi)$ and $\Im(\pi)=M_l(B)$. $\Im(T_{\aa})\subset \Ker(\pi)$ is obvious. We prove the reverse inclusion. If $\pi(\xx) = 0 \in M_l(B)\ (\xx\in \Ker(F^{(\lambda)})(B))$, then we have $\xx\in\Im(T_{\aa})$. Hence there exists a $\zz\in W_A(B)$ such that $\xx=T_{\aa} (\zz)$. Then the equalities $F^{(\lambda)}(\xx) = F^{(\lambda)}\circ T_{\aa} (\zz) = 0$ hold. Therefore, by Lemma~4.7, we obtain
\begin{align*}
\zz \in \Ker(F^{(\lambda)} \circ T_{\aa})(B)=\Ker(F^{(\lambda^{p^l})})(B).
\end{align*}
Next, we prove the surjectivity of $\pi$. Take $\xx\in W_A(B)$ such that $\overline{\xx}\in M_l(B)$. This means that $F^{(\lambda)}(\xx)\in\Im(T')(B)$. Hence there exists $\ww\in W_A(B)$ such that $T'(\ww)=F^{(\lambda)}\circ T_{\aa}(\ww)=F^{(\lambda)}(\xx)$. This means that $T_{\aa}(\ww)-\xx\in\Ker(F^{(\lambda)})(B)$. Therefore $\pi(T_{\aa}(\ww)-\xx)=\overline{\xx}$. Hence $\pi$ is surjective.
\end{proof}

Hereafter, we assume that $\nu_0$ and $a_0$ are units or $0$. Then we can use Corollary~1.7. Now, by combining the exact sequences $(1.3),\ (4.10)$ and the isomorphism $(4.8)$, we have the following diagram consisting of exact horizontal lines and vertical isomorphisms except $\phi$:
\begin{align}
\begin{CD}
\Hom(\widehat{\gg}^{(\lambda^{p^l})},\widehat{\GG}_{m,B}) @>{(\psi^{(l)})^\ast}>> \Hom(\widehat{\gg}^{(\lambda)},\widehat{\GG}_{m,B}) @>{(\iota)^\ast}>> \Hom(N_l,\widehat{\GG}_{m,B}) @>>> 1\\
@A{\phi_3}AA @A{\phi_4}AA  @A{\phi}AA @AAA\\
\Ker(F^{(\lambda^{p^l})})(B) @>{T_{\aa}}>> \Ker(F^{(\lambda)})(B) @>{\pi}>> M_l(B) @>>> 0,
\end{CD}
\end{align}
where $\phi$ is the following homomorphism induced from the exact sequence $(1.2)$ and the isomorphism $(4.8)$:
\begin{align*}
\phi:M_l(B) \rightarrow \Hom(N_l,\widehat{\GG}_{m,B});\ \overline{\vv} \mapsto E_p(\overline{\vv},\lambda;\overline{X}):=E_p(\vv,\lambda;\overline{X}).
\end{align*}

\begin{lemma}
The morphism $\phi$ is well-defined.
\end{lemma}

\begin{proof}
For $\overline{\vv} \in M_l(B)$, we choose an inverse image $\vv+T_{\aa}(\uu) \in W(B)$, where $\vv \in \Ker(F^{(\lambda)})(B)$ and $\uu\in \Ker(F^{(\lambda^{p^l})})(B)$. By Lemma~4.2, we have
\begin{align*}
E_p(\overline{\vv},\lambda;X)=E_p(\vv,\lambda;X)\cdot E_p(T_{\aa}(\uu),\lambda;X)
                      &= E_p(\vv,\lambda;X)\cdot E_p(\uu,\lambda;\psi^{(l)}(X))\\
                      &\equiv E_p(\vv,\lambda;X)\ \mbox{($\mod\ \psi^{(l)}(X)$)}.
\end{align*}
This means that $\phi$ is well-defined.
\end{proof}

If the diagram $(4.11)$ is commutative, by using the five lemma, $\phi$ becomes an isomorphism, i.e., $M_l(B) \simeq \Hom(N_l,\widehat{\GG}_{m,B})$.

\begin{lemma}
The diagram {\rm (4.11)} is commutative. Hence $\phi$ is an isomorphism.
\end{lemma}

\begin{proof}
For $(\psi^{(l)})^\ast \circ \phi_1 = \phi_2 \circ T_{\aa}$, we must show the equality
\begin{align*}
E_p(\uu,\lambda^{p^l};\psi^{(l)}(X)) = E_p(T_{\aa}(\uu),\lambda;X),
\end{align*}
where $\uu\in \Ker(F^{(\lambda^{p^l})})(B)$. This is already shown in proof of Lemma~4.7. The equality of $(\iota)^\ast \circ \phi_2 = \phi \circ \pi$ follows from the definition of $\phi$. It is obvious that the last square is commutative.
\end{proof}

\begin{remark}
{\rm
Note that Theorem~1.4 plays a decisive role in making the commutative diagram (4.11) holds. This is the missing part in \cite{A2}.
}
\end{remark}

By the above arguments, we obtain $M_l(B) \simeq \Hom(N_l,\widehat{\GG}_{m,B})$ for any $A$-algebra $B$. Note that, in general, $N_l$ is a formal group scheme, not a finite group scheme. In order for $N_l$ to be a finite group scheme, we need the following Lemma.

\begin{lemma}
Let $A_n$ be a $\ZZ/(p^n)$-algebra and $\nu_k$'s be divided by $p$. Then
\begin{align*}
N_l=\Ker[\psi^{(l)}:\widehat{\gg}^{(\lambda)}\to\widehat{\gg}^{(\lambda^{p^l})}]
\end{align*}
is a finite group scheme of order $p^l$.
\end{lemma}

\begin{proof}
The coordinate ring $N_l$ is given by $A_n[[X]]/(\psi^{(l)}(X))$. Then we have
\begin{align*}
\overline{X}^{p^l}=-\alpha_1\overline{X}-\alpha_2\overline{X}^2-\cdots-\alpha_{p^l-1}\overline{X}^{p^l-1}\in A_n[[X]]/(\psi^{(l)}(X)).
\end{align*}
Since $\nu_k$'s are nilpotent elements, $\alpha_i$'s are also nilpotent elements. Hence the class $\overline{X}$ is nilpotent element (\cite[Chap.~1, Ex.~2]{AT}). In this case, the kernel of $\psi^{(l)}$ has the equalities
\begin{align*}
N_l=\Ker(\psi^{(l)})=\Spf(A_n[[X]]/(\psi^{(l)}(X)))=\Spec(A_n[X]/(\psi^{(l)}(X))).
\end{align*}
Moreover, since $A_n[X]/(\psi^{(l)}(X))$ is a free module of rank $p^l$, $N_l$ is a finite group scheme of order $p^l$. 
\end{proof}

Let $A_n$ be a $\ZZ/(p^n)$-algebra and $B$ be an $A_n$-algebra. Assume that $\nu_0=0$ and $\nu_k$ is divisible by $p$ for $1\leq k\leq l-1$. Then, since Corollary~1.7 holds over $B$, we can use Lemma~4.11. Hence we get the isomorphism $M_l(B) \simeq \Hom(N_l,\widehat{\GG}_{m,B})$, where $N_l$ is a finite group scheme of order $p^l$ by Lemma~4.13. We denote by $N_l^D$ the Cartier dual of $N_l$. Thus, by the fppf-sheafification, we have
\begin{align*}
\widetilde{N_l^D}=N_l^D\simeq\widetilde{M}_l={\mathrm{Ker}}[ \widetilde{F^{(\lambda)}} : \widetilde{W_{A_n} / T_{\aa}} \rightarrow \widetilde{W_{A_n} / T'} ]
\end{align*}
in the category of sheaves on $A_n$-algebras. Hence we obtain Theorem~1.9.

\section*{acknowledgements}
The author would like to thank Professor Akira Masuoka for his kind advice and suggestions. He sincerely thanks the people of Meisei University for their hospitality. He is also grateful to the referee for a number of suggestions to improve the paper.

\bibliographystyle{amsplain}
\bibliography{ref_amano.bib}

\providecommand{\bysame}{\leavevmode\hbox to3em{\hrulefill}\thinspace}
\providecommand{\MR}{\relax\ifhmode\unskip\space\fi MR }
\providecommand{\MRhref}[2]{%
  \href{http://www.ams.org/mathscinet-getitem?mr=#1}{#2}
}
\providecommand{\href}[2]{#2}
\begin{thebibliography}{10}

\bibitem{A1}
M.~Amano, \emph{{\rm On the Cartier duality of certain finite group schemes of
  order $p^n$}}, Tokyo J. Math. \textbf{33} (2010), 117--127.

\bibitem{A2}
M.~Amano, \emph{{\rm On the Cartier duality of certain finite group schemes of
  order $p^n$,~II}}, Tsukuba J. Math. \textbf{37} (2013), no.~2, 259--269.

\bibitem{A3}
\bysame, \emph{{\rm Corrigendum to ``On the Cartier duality of certain finite
  group schemes of order~$p^n$,~II'' [Tsukuba J. Math. 37 (2) (2013)
  259--269]}}, Tsukuba J. Math. \textbf{41} (2017), no.~1, 167--168.

\bibitem{AT}
M.~F. Atiyah and I.~G. Macdonald, \emph{{\rm Introduction to commutative
  algebra}}, Addison-Wesley, Reading, Mass, 1969.

\bibitem{DG}
M.~Demazure and P.~Gabriel, \emph{{\rm Groupes alg\'{e}briques}}, Tome I,\
  Masson-North-Holland, Paris-Amsterdam, 1970.

\bibitem{HZ}
M.~Hazewinkel, \emph{{\rm Formal~groups~and~applications}}, Academic~Press,
  New~York, 1978.

\bibitem{MRT1}
A.~M\'{e}zard, M.~Romagny, and D.~Tossici, \emph{{\rm Models of group schemes
  of roots of unity}}, Ann. Inst. Fourier, Grenoble \textbf{63} (2013), no.~3,
  1055--1135.

\bibitem{MRT2}
\bysame, \emph{{\rm Sekiguchi-Suwa theory revisited}}, Journal de Th\'{e}orie
  des Nombres de Bordeaux \textbf{26} (2014), 163--200.

\bibitem{SOS}
F.~Oort, T.~Sekiguchi, and N.~Suwa, \emph{{\rm On the deformation of
  Artin-Schreier to Kummer}}, Ann. Sci. \'{E}cole Norm. Sup. \textbf{22}
  (1989), no.~4, 345--375.

\bibitem{SS2}
T.~Sekiguchi and N.~Suwa, \emph{{\rm On the unified Kummer-Artin-Schreier-Witt
  theory}}, Pr\'{e}publication No.~111, Universit\'{e} de Bordeaux (1999).

\bibitem{SS1}
\bysame, \emph{{\rm A note on extensions of algebraic and formal groups IV}},
  Tohoku Math. J \textbf{53} (2001), 203--240.

\bibitem{To}
D.~Tossici, \emph{{\rm Models of $\mu_{p^2,K}$ over discrete valuation ring}},
  Journal of Algebra \textbf{323} (2010), 1908--1957.

\bibitem{W}
W.~Waterhouse, \emph{{\rm A Unified Kummer-Artin-Schreier Sequence}},
  Mathematische Annalen \textbf{277} (1987), 447--452.

\bibitem{WW}
W.~Waterhouse and B.~Weisfeiler, \emph{{\rm One-dimensional affine group
  schemes}}, Journal of Algebra \textbf{66} (1980), 555--568.

\end{thebibliography}

\end{document}